\documentclass[11pt]{article}
\usepackage[pagebackref,colorlinks=true,pdfpagemode=none,urlcolor=blue,linkcolor=blue,citecolor=blue]{hyperref}

\usepackage{amsthm, amssymb, bm}
\usepackage{soul, color}
\usepackage[]{amsmath}
\usepackage[]{amsfonts}
\usepackage[]{fancyhdr}
\usepackage[]{graphicx}
\usepackage[mathscr]{euscript}

\graphicspath{{/EPSF/}{../figures/}{figures/}}

\newtheorem{theorem}{Theorem}[]
\newtheorem{lemma}[theorem]{Lemma}

\newtheorem{definition}[theorem]{Definition}
\newtheorem{remark}[theorem]{Remark}

\def \Cm {\mathbb{C}}

\def \Rm {\mathbb{R}}
\def \Sm {\mathbb{S}}

\def \Zm {\mathbb{Z}}

\def\M{\mathcal{M}}

\newcommand{\cout}[1]{}

\newcommand{\x}{\mathrm{x}}

\newcommand{\dbar}{\overline{\partial}}

\newcommand{\hd}{H^\text{sol}}

\newcommand{\sobzero}{W^{1,2}_0}
\newcommand{\sob}{W^{1,2}}

\newcommand{\id}{ {\text{id}}}

\title{Inverse source problems in transport via attenuated tensor tomography}
\author{Guillaume Bal\thanks{Departments of Statistics and Mathematics, University of Chicago, Chicago IL 60637; email: \hbox{guillaumebal@uchicago.edu}} \and Fran\c{c}ois Monard\thanks{Department of Mathematics, University of California, Santa Cruz CA 95064; email: fmonard@ucsc.edu}}

\begin{document}

\maketitle

\begin{abstract} We establish results for the injectivity and injectivity modulo gauge of certain inverse source problems in transport on a simply connected domain with variable index of refraction inducing a 'simple geometry'. The model given by radiative transfer involves a scattering kernel with finite harmonic content in the deviation angle. The results on injectivity are constructive, and they are connected to the explicit inversion (modulo kernel) of the attenuated X-ray transform on tensor fields on simple Riemannian surfaces. 
\end{abstract}

\section{Introduction} 

We consider an inverse problem associated with the propagation of photons with attenuation and scattering effects, in a two-dimensional medium with variable index of refraction. The domain $M$ is simply connected in $\Rm^2$, endowed with a Riemannian metric $g = c^{-2}(\x) \id$ where $c(\x)$ corresponds to the local speed induced by a variable index of refraction. In this case, the dynamics takes place in the unit phase space (or unit tangent bundle)
\begin{align*}
    SM = \{(\x,v)\in TM,\ g_\x(v,v) = 1\},
\end{align*}
with ingoing/outgoing boundaries 
\begin{align}
  \Gamma_\pm := \{ (\x,v) \in SM,\ \x\in \partial M,\ \pm g_\x(v,\nu_\x) >0 \},
  \label{eq:gammapm}
\end{align}
where $\nu_\x$ denotes the unit outward normal at $\x\in \partial M$. Linear propagation is modeled by the geodesic flow of $g$, $\varphi_t:SM\to SM$, and we assume that the flow is {\em simple} (see Sec. \ref{sec:notation} below). Representing a unit tangent vector as $v = c(\x) \binom{\cos\theta}{\sin\theta}\in S_\x M$, the infinitesimal generator of the linear transport is given by 
\begin{align}
    X = c(\x) (\cos\theta \partial_x + \sin\theta \partial_y) + (-\partial_y c \cos \theta  + \partial_x c \sin\theta ) \partial_\theta.
    \label{eq:X}
\end{align}
Attenuation effects are modeled via a position dependent function $a(\x)$, and scattering is defined through a collision kernel $k(\x,\theta,\theta')$. The propagation of the density of photons is then modeled by the following PDE posed in $SM$, to be solved for the photon density $u(\x,\theta)$:   
\begin{align}
    X u + a (\x) u = \frac{1}{2\pi} \int_{\Sm_x} k(\x, \theta, \theta') u(\x,\theta')\ d\theta' + f(\x,\theta) \qquad (SM), \qquad \quad u|_{\Gamma_-} = h,
  \label{eq:transport}
\end{align}
in a sense made precise in Section \ref{sec:forward}. The 'inputs' $h$ and $f$ generate the solution, and under admissibility conditions on the metric $g$ and the optical parameters $(a, k)$, one may define uniquely $u|_{\Gamma_+}$ as a function of $(h,f)$, allowing to formulate several inverse problems: 

(P1) In the case $f = 0$, the so obtained boundary-to-boundary map is called the {\em albedo operator} ${\mathcal A}_{a,k} \colon L_\mu^1(\Gamma_-)\to L_\mu^1(\Gamma_+)$, a 'transmission' setting where the typical inverse question is whether one can reconstruct $(a,k)$ from the albedo operator $\mathcal A_{a,k}$. Such a problem has been well studied in the literature, in Euclidean settings \cite{Choulli1999,Choulli1996,Stefanov2009a,Bal2010b,Bal2008b}, on Riemannian manifolds (modeling media with variable index of refraction) \cite{McDowall2009,McDowall2005,McDowall2004,McDowall2010,McDowall2011,McDowall2010a}, with a recent generalization to the case where a magnetic field is present \cite{Assylbekov2015}.  

(P2) In the case $h=0$, the setting covered in what follows, the internal source $f$ generates the outward radiated field $u|_{\Gamma_+}$, an 'emission' setting with applications to SPECT and Optical Molecular Imaging, where the inverse problem is to reconstruct the source $f(\x,v)$ from $u|_{\Gamma_+}$. In the absence of scattering (i.e., $k\equiv 0$), this is equivalent to inverting the so-called {\em attenuated X-ray transform} of $f$, which in the case where $f$ is polynomial in the $v$ variable, can also be viewed as the {\em attenuated tensor tomography problem}. The study of the attenuated X-ray transform has advanced quite a bit over the past few years \cite{Salo2011,Paternain2011a,Monard2015,Assylbekov2017,Monard2017a}, an in turn motivates the present results. The main literature on the inverse source problem is when the source $f$ is isotropic: in Euclidean cases when scattering is small, Bal and Tamasan proposed a Neumann-series type inversion in \cite{Bal2007}, based on the inversion of the attenuated X-ray transform over functions; for general scattering, the problem was proved to be Fredholm (injective up to a finite-dimensional space of obstructions) in \cite{Stefanov2008a}; in Riemannian settings, Sharafutdinov proved a stability estimate using energy identities (in particular proving injectivity) under a no-conjugate point assumption involving an effective curvature combining the curvature of the manifold as well as the optical parameters, in turn implying a smallness assumption on the optical parameters see \cite{Sharafutdinov1994,Sharafutdinov1999}. 

In the present work, we are concerned with the inverse source problem (P2), where the setting is a medium with variable index of refraction, whose linear transport induces a 'simple' flow. Compared to earlier results, the smallness assumption on $k$ is dropped and replaced by the condition that $k$ has finite harmonic content in the angular deviation variable. We further consider sources with polynomial angular dependence, for which the corresponding inverse problem may or may not be injective, and in the case where it is not, we fully characterize the extent of non-injectivity. In the injective cases, the approach is constructive, and as explicit as the solution to the attenuated tensor tomography is. In particular, in the case of the Euclidean disk, a complete answer involving explicit reconstruction formulas is provided in \cite{Monard2017a}, while for simple surfaces, a partially explicit inversion is proposed in \cite{Krishnan2018} (the missing point being an explicit construction of special geodesically invariant distributions). 

The main tool is to make use of the recent results for the {\em attenuated tensor tomography problem} in \cite{Monard2017a,Krishnan2018}. Specifically, for the attenuated transform defined over tensor fields (or equivalently, over functions of $(\x,v)$ with polynomial dependence on $v$), while the problem admits a non-trivial gauge, one is able to devise a representative to be reconstructed as well as an explicit and efficient procedure to reconstruct it. In other words, we are fixing a gauge for the attenuated transform, where the dependency of the reconstructed candidate in terms of the unknown source term $f(\x,v)$ and the initial transport solution $u$ is fully understood. We are then able to extract $f$ (entirely if $f$ is isotropic, partially otherwise) from this reconstructed candidate. The results presented are specifically two-dimensional, as they rely heavily on complex-analytic tools, and the three-dimensional case would also be of great interest. 

Finally, the authors were recently made aware of the recent work \cite{Fujiwara2019}, also addressing the inverse source problem in transport. A similar feature is that both articles exploit the finiteness in harmonic content of the scattering kernel, in a crucial way, to reconstruct a source term, though using different methods in different contexts (the theory of A-analytic functions used in \cite{Fujiwara2019} restricts the analysis to the case of constant index of refraction, which on the other hand allows for a good understanding of stability properties). The present work also addresses anisotropic sources.  

We now present the main results and give an outline of the remainder at the end of the next section.

\section{Statement of the main results} \label{sec:main}

\subsection{Notation} \label{sec:notation}

Recall that we denote the phase space $SM := \{(\x,v)\in TM,\ g_\x(v,v) = 1\}$, with boundary 
\begin{align*}
    \partial SM = \{(\x,v)\in SM,\ \x\in \partial M\} = \Gamma_+ \sqcup \Gamma_- \sqcup \Gamma_0,
\end{align*}
where $\Gamma_\pm$ are defined in \eqref{eq:gammapm} and $\Gamma_0$ (the 'glancing' boundary), is defined through the condition $g(\nu_x,v) = 0$. We denote $\pi\colon SM\to M$ the canonical projection and for any $(\x,v)\in SM$, we define $\tau(\x,v)$ the smallest non-negative time $t$ such that $\pi(\varphi_t(\x,v)) \in \partial M$. Thoughout the article, we will assume that $(M,g)$ is {\em strictly convex}\footnote{in the sense that $\partial M$ has positive definite second fundamental form.} and {\em non-trapping}\footnote{in the sense that $\tau(\x,v) <\infty$ for all $(\x,v)\in SM$}. The inverse results will further assume that $(M,g)$ be {\em simple}, in the sense that it is strictly convex, non-trapping, and has no conjugate points.

The first two requirements on $M$ imply that $M$ is simply connected. Therefore we will not restrict generality by employing a global 'isothermal' chart $(x,y,\theta) = (\x,\theta)$ of $SM$ where the metric looks like $g = c^{-2}(\x)\id$ for some fixed scalar 'speed' function $c(\x)>0$, and where a tangent vector $v$ is uniquely described by the angle $\theta$ through the correspondence $v = c(\x) \binom{\cos\theta}{\sin \theta}$. In these coordinates, the geodesic vector field $X$ takes expression \eqref{eq:X}, and $SM$ carries the 'Liouville' volume form $d\Sigma^3 := c^{-2}(\x)\ d\x\ d\theta$ (with $d\x$ the Euclidean area). In what follows, for $p\in [1,\infty]$, one may denote $L^p(SM)$ spaces with respect to that volume. The inward/outward boundaries $\Gamma_\pm$ also inherit the surface measure $d\Sigma^2 = ds\ d\theta$, where $ds$ denotes $g$-arclength along the boundary $\partial M$. For $(\x,v) \in \partial SM$, we let $\mu(\x,v) := g_\x(\nu_\x,v)$, and will denote $L^p(\Gamma_\pm) := L^p(\Gamma_\pm, d\Sigma^2)$, as well as $L^p_\mu (\Gamma_\pm) := L^p(\Gamma_\pm, \pm \mu\ d\Sigma^2)$. Naturally, $L^p(\Gamma_\pm) \subsetneq L^p_\mu (\Gamma_\pm)$.

\paragraph{Fourier analysis on $SM$.} Via Fourier series in $\theta$, functions in $L^2(SM)$ may be decomposed into circular harmonics 
\begin{align}
    L^2(SM) = \bigoplus_{k=0}^\infty H_k, \qquad H_k := \ker(\partial_\theta^2 + k^2 Id).
    \label{eq:harmonics}
\end{align}
The subspace $H_0$ is isometric to $L^2(M)$ and for $k\ge 1$, $f\in H_k$ if and only if $f = f_{k,+} + f_{k,-}$ with $f_{k,\pm}(\x,\theta) = \tilde f_{k,\pm} (\x) e^{\pm i k\theta}$ for some functions $\tilde f_{k,\pm}\in L^2(M)$ (we write this splitting $H_k = E_k \oplus E_{-k}$ for $k>0$). A function $u\in L^2(SM,\Cm)$ decomposes accordingly
\begin{align*}
    u(\x,\theta) &= \sum_{k=0}^\infty u_k(\x,\theta) = u_0(\x) + \sum_{k=1}^\infty (\tilde u_{k,+}(\x)e^{ik\theta} + \tilde u_{k,-}(\x) e^{-ik\theta}), \\
    \|u\|^2 &= 2\pi \left( \|u_0\|_M^2 + \sum_{k=1}^\infty (\|\tilde u_{k,+}\|_M^2 + \|\tilde u_{k,-}\|_M^2) \right). \qquad (\text{Parseval})
\end{align*}
Since the notation $H$ is reserved for subspaces of $L^2(SM)$, Sobolev-type spaces will be denoted using $\sob (H_k)$ for the space of elements in $H_k$ whose spatial components lie in the usual Sobolev space $\sob(M)$ ($L^2$ functions with gradient in $L^2$). Similarly, we will denote the $\sob(M)$-closure of $C_c^\infty(M)$ by $\sobzero(M)$. 

Finally, we will say that $u$ is of degree $m$ if in the decomposition \eqref{eq:harmonics}, $u_k = 0$ for all $k>m$.

\subsection{Main results: applications of attenuated tensor tomography to the inverse source problem}

Consider the equation
\begin{align}
    X u + a(\x) u = Su(\x,\theta) + f(\x, \theta) \qquad (SM), \qquad u|_{\Gamma_-} = 0, 
    \label{eq:boltzmann}
\end{align}
where the scattering term takes the form $Su(\x,\theta) := \frac{1}{2\pi} \int_{\Sm^1} k(\x,\theta-\theta') u(\x,\theta')\ d\theta'$. In the transport litterature, $a = \sigma_p + \sigma_a$, where $\sigma_a(\x)$ represents pure loss due to absorption and $\sigma_p (\x) = \frac{1}{2\pi} \int_{\Sm^1} k(\x,\theta)\ d\theta = k_0 (\x)$. We say that $(a,k)$ is {\em admissible} if they satisfy
\begin{align}
    \begin{split}
	0\le a\in L^\infty(M) \quad \text{and} \quad 0\le k (\x,\cdot) \in L^1(S_\x M)\quad \text{for a.e. } \x\in M.
    \end{split}
    \label{eq:admissibility}
\end{align}
In addition, we define the {\em subcriticality} condition as
\begin{align}
    \sigma_a (\x) \ge \delta>0 \qquad \text{for some positive constant } \delta.
    \label{eq:subcrit}
\end{align}
We first show in Section \ref{sec:forward} that under conditions \eqref{eq:admissibility}-\eqref{eq:subcrit}, Equation \eqref{eq:boltzmann} has a unique solution $u$, whose outgoing trace allows to define our measurement operator.
\begin{theorem}\label{thm:fwd}
    Suppose that $(M,g)$ is convex and non-trapping, and let optical parameters $a, k$ satisfy conditions \eqref{eq:admissibility}-\eqref{eq:subcrit}. Then the operator 
    \begin{align}
	\M_{a,k} \colon L^2(SM) \to L^2(\Gamma_+), \qquad \M_{a,k} f := u|_{\Gamma_{+}},\qquad \text{where } u \text{ solves } \eqref{eq:boltzmann},
	\label{eq:measurement}
    \end{align}
    is well-defined and continuous, with norm not exceeding $\sqrt{C_0} \left( \frac{Q_\infty}{\delta} + 1 \right)$, where $C_0$ is a convexity constant defined in \eqref{eq:C0}, and $Q_\infty := \sup_{\x\in M} a(\x) + \sup_{\x\in M} \int_{S_\x M} k(\x,\alpha)\ d\alpha$.  
\end{theorem}

It is then natural to consider the {\em inverse source problem}, which consists of reconstructing the source term $f$ from the measurements $\M_{a,k} (f)$. In the absence of scattering, such a problem is nothing but the inversion of the attenuated X-ray (or Radon) transform $I_{\sigma_a}$. In this case, $f$ cannot be reconstructed uniquely unless the dependency of $f$ in $(\x,\theta)$ is reduced, and some typical examples where injectivity holds is for $f = f_0(\x) + X_\perp f_\perp$ or $f(\x,\theta) = f_1(\x) e^{i\theta} + f_{-1} (\x) e^{-i\theta}$ (the case of vector fields). If $f$ is direction-dependent but with finite degree, the attenuated X-ray transform still has a well-understood kernel. Now, in the presence of scattering and if both $f$ and $k$ have finite degree, with $k$ of the form
\begin{align*}
    k(\x,\alpha) = \sum_{n=-m}^m \tilde k_n(\x) e^{in\alpha}, 
\end{align*}
then equation \eqref{eq:boltzmann}, upon representing $u(\x,\theta)$ into a Fourier series, takes the form 
\begin{align*}
  X u + a(\x) u = \sum_{n=-m}^m \tilde k_n(\x) \tilde u_n(\x) e^{in\theta} + f(\x,\theta). 
\end{align*}
In particular, the right-hand side has finite degree, and $\M_{a,k} f = I_a [Su + f]$ lies in the range of the attenuated X-ray transform defined over functions of high enough degree. The goal is then to reconstruct $f$ (or a gauge representative of it) by removing the effects of the scattering term $Su$. 

As in the absence of scattering, this problem may have a large kernel for general direction-dependent sources. We therefore present two results corresponding to cases where the problem is injective, and cases where the problem has a gauge. In the latter, we fully characterize the gauge. In either case, the frequency content of $k$ makes it necessary to work with the attenuated X-ray transform on tensor fields. Since it has a large kernel there, to reconstruct a source, one must $(i)$ make a judicious choice of which representative is being reconstructed, and $(ii)$ extract $f$ from it. For the statement below, we define
\begin{align}
    X_\perp := c(\x) (\sin\theta \partial_x - \cos\theta \partial_y) - (\partial_x c \cos \theta + \partial_y c \sin \theta) \partial_\theta.
    \label{eq:Xperp}
\end{align}

\begin{theorem}\label{thm:main1} Let $(M,g)$ be a simple domain and $(a,k)$ a pair of optical parameters satisfying \eqref{eq:admissibility}-\eqref{eq:subcrit}, with $k$ of finite degree and $a$ smooth. Suppose that the source $f$ is of either of the following forms: 
    
    (1) $f = f_0 + X_\perp f_\perp$ for $f_0 \in L^2(M)$ and $f_\perp \in \sob(M)$. 

    (2) $f = f_1 = f_{1,+}(\x) e^{i\theta} + f_{1,-}(\x) e^{-i\theta}$ (vector field) for $f_{1,\pm}\in L^2(M)$. 
    
    \noindent Then the mapping $f\mapsto \M_{a,k} (f)$ is injective and constructively invertible. 
\end{theorem}

Case (1) above with $f_\perp = 0$ corresponds to isotropic sources and is the most common example in, e.g. Optical Molecular Imaging. Case (2) is applicable to Doppler Tomography \cite{Sparr1995}. Finally, we state a second result for the case of more general sources, where the problem now has a gauge. The main result is a full characterization of the gauge.

\begin{theorem}\label{thm:main2} Let $(M,g)$ be a simple domain, $(a,k)$ a pair of optical parameters satisfying \eqref{eq:admissibility}-\eqref{eq:subcrit}, with $k$ of finite degree and $a$ smooth. Suppose that the source $f$ has degree at most $m\ge 1$. Then $\M_{a,k}(f) = 0$ if and only if there exists $p$ of degree $m-1$ with components in $\sobzero(M)$ such that 
    \begin{align*}
	f = Xp + ap - Sp.
    \end{align*}
\end{theorem}

\begin{remark} In both theorems above, the smoothness assumption on $a$ is mainly due to Theorem \ref{thm:KMM} below, used in our main results. Smoothness on $a$ is required by microlocal arguments in \cite{Salo2011} in order to construct holomorphic integrating factors for attenuated X-ray transforms. However, it is expected that the same results hold using less regularity on $a$, e.g. $a\in C^0(M)$ is expected to suffice.    
\end{remark}

\paragraph{Outline.} The remainder of the article is structured as follows. We first describe in Sec. \ref{sec:forward} the forward theory that is necessary to properly define the measurement operator $\M_{a,k}$. This includes an $L^2$ treatment of trace issues (\S \ref{sec:spacestraces}) and forward solvability, followed by the proof of Theorem \ref{thm:fwd} (\S \ref{sec:pfTh1}). Section \ref{sec:reconstruction} then covers the main 'inverse' results, starting with geometric preliminaries (\S \ref{sec:prelims}) and recalls from attenuated tensor tomography (\S \ref{sec:atRt}), followed by the proofs of Theorems \ref{thm:main1} (\S \ref{sec:main1}) and \ref{thm:main2} (\S \ref{sec:main2}).

\section{Forward mapping properties} \label{sec:forward}

The main purpose of this section is to prove Theorem \ref{thm:fwd}, providing along the way a self-contained $L^2$ theory of transport in our setting. The forward results will be formulated for indices of refraction such that the domain is strictly convex, non-trapping (in other words, assuming the absence of conjugate points is not necessary in this section). 

\subsection{Spaces and traces} \label{sec:spacestraces}

We denote $\|\cdot\|$ the norm in $L^2(SM)$, and define the Hilbert space
\begin{align*}
    W^{1,2}_X (SM) := \{ u\in L^2(SM), \quad Xu \in L^2(SM) \}, \qquad \|u\|^2_{2,X} := \|u\|^2 + \|Xu\|^2.
\end{align*}
Using the strict convexity of the boundary, \cite[Lemma 4.1.2]{Sharafudtinov1994} implies that there exists a constant $C_0$ such that 
\begin{align}
    \tau(\x,v) \le C_0\ \mu(\x,v), \qquad \forall (\x,v)\in \Gamma_-.
    \label{eq:C0}
\end{align}
Green's formula for $X$ reads
\begin{align}
    \int_{SM} Xu\ d\Sigma^3 = \int_{\partial SM} u\ \mu\ d\Sigma^2, \qquad u\in C^\infty(\overline{SM}), 
    \label{eq:Green}
\end{align}
and upon applying this to the solution $u$ of $Xu = -f$ with $u|_{\Gamma_-} = 0$, we obtain Santal\'o's formula
\begin{align}
    \int_{SM} f\ d\Sigma^3 = \int_{\Gamma_+} \int_{-\tau(\x,-v)}^0 f(\varphi_t(\x,v))\ dt\ \mu\ d\Sigma^2.
    \label{eq:Santalo}
\end{align}

\begin{lemma}\label{lem:trace}
    Suppose $(M,g)$ strictly convex and non-trapping, and let $C_0>0$ be the constant in \eqref{eq:C0}. The operator $T \colon C^\infty(\overline{SM}\backslash \Gamma_0) \to C^\infty(\Gamma_+)$ defined by 
    \begin{align*}
	Tu(\x,v) := u(\x,v) - u (\varphi_{-\tau(\x,-v)}(\x,v)), \qquad (\x,v) \in \Gamma_+,
    \end{align*}
    extends as a bounded map $T\colon W_X^{1,2}(SM) \to L^2(\Gamma_{+})$, with norm at most $\sqrt{C_0}$.
\end{lemma}

\begin{proof} 
    We prove it for smooth $u$ and extend by density. For such an $u$ and for all $(\x,v)\in \Gamma_+$, the Fundamental Theorem of Calculus gives
    \begin{align*}
	Tu(\x,v) = u(\x, v) - u (\varphi_{-\tau(\x,-v)}(\x,v)) = \int_{-\tau(\x,-v)}^0 Xu (\varphi_t(\x,v))\ dt.
    \end{align*}
    By the Cauchy-Schwarz inequality, we have
    \begin{align*}
	|Tu(\x,v)|^2 \le \tau(\x,-v) \int_{-\tau(\x,-v)}^0 |Xu|^2(\varphi_t(\x,v))\ dt. 
    \end{align*}
    Therefore, we obtain
    \begin{align*}
	\int_{\Gamma_+} |Tu(\x,v)|^2\ d\Sigma^2 &\le \int_{\Gamma_+} \tau(\x,-v) \int_{-\tau(\x,-v)}^0 |Xu|^2(\varphi_t(\x,v))\ dt\ d\Sigma^2 \\
	&\stackrel{\eqref{eq:C0}}{\le} C_0 \int_{\Gamma_+} \int_{-\tau(\x,-v)}^0 |Xu|^2(\varphi_t(\x,v))\ dt\ \mu\ d\Sigma^2 \stackrel{\eqref{eq:Santalo}}{=} C_0 \|Xu\|^2,
    \end{align*}
    hence the result.     
\end{proof}

\begin{remark}\label{rem:notrace}
    Lemma \ref{lem:trace} does not imply that the trace operators $u|_{\Gamma_\pm}$ separately extend in a similar fashion\footnote{See also \cite[\S\ 'Orientation', p. 219]{Dautray6} where the same claim is made in the case of constant refractive index.}, let alone with the (larger) target space $L^2_\mu(\Gamma_\pm)$. Indeed, on any convex, non-trapping domain, consider the exit time function $\tau\colon \Gamma_- \to \Rm_+$, and denote $\tau_\psi$ its extension by constancy along the flow $\varphi_t$, in particular, $\tau_\psi\in C^\infty(\overline{SM})\backslash \Gamma_0$ and $X \tau_\psi =0$. For any $\eta\in \Rm$, we consider the quantities 
    \begin{align*}
	\|(\tau_\psi)^\eta \|^2_{L^2_\mu(\Gamma_-)} = \int_{\Gamma_-} \tau^{2\eta} \mu\ d\Sigma^2, \qquad \|(\tau_\psi)^\eta\|^2_{W_X^{1,2}(SM)} = \int_{\Gamma_-} \tau^{2\eta+1} \mu\ d\Sigma^2,
    \end{align*}
where the second quantity is computed using \eqref{eq:Santalo}. The first quantity is finite iff $2\eta + 1>-1$ (or $\eta>-1$) and the second is finite iff $2\eta+2>-1$ (or $\eta>-3/2$). Thus for any $\eta\in (-3/2,-1]$, the function $(\tau_\psi)^{\eta}$ belongs to $W_X^{1,2}(SM)$ but $(\tau_\psi)^{\eta}|_{\Gamma_-}$ does not belong to $L^2_\mu(\Gamma_-)$.
\end{remark}

\paragraph{Traces.} In spite of Remark \ref{rem:notrace}, traces do extend continuously on $L^2(K, d\Sigma^2)$ for $K$ any compact subset of $\Gamma_\pm$. 

\begin{lemma}\label{lem:Ktrace} Suppose $(M,g)$ is strictly convex and non-trapping. Let $K$ be a compact subset of $\Gamma_+$ (resp. $\Gamma_-$). Then the trace mapping $C^\infty(\overline{SM}) \ni u\mapsto u|_K \in C^\infty(K)$ extends by continuity to a continuous mapping of $W_X^{1,2}(SM) \to L^2(K)$.     
\end{lemma}

\begin{proof} We treat the case where $K\subset \Gamma_+$, as the case $K\subset \Gamma_-$ is deduced from it by considering the function $u(\x,-v)$. Fix $K$ a compact subset of $\Gamma_+$, and let $h\in C_c^\infty(\Gamma_+)$ a non-negative function equal to $1$ on $K$. We extend $h$ by constancy along the flow $\varphi_t$, into a function $h_\psi \in C^\infty(\overline{SM})$, vanishing in a neighborhood of $\Gamma_0$. 
    Also define $\chi \in C^\infty([0,1], [0,1])$ equal to $1$ on $[0,1/3]$ and $0$ on $[2/3,1]$. Then the function $\overline{SM} \ni (\x,v) \mapsto p(\x,v) := \chi \left( \frac{\tau(\x,v)}{ \tau(\x,v) + \tau(\x,-v)} \right)$ is smooth away from $\Gamma_0$, thus the product $\ell:= h_\psi p$ defines a smooth function on $\overline{SM}$. Now let $u\in C^\infty(\overline{SM})$. Then $u\ell$ vanishes on $\Gamma_-$ so that $T(u\ell) = u\ell|_{\Gamma_+}$, and $u\ell|_{\Gamma_+}$ agrees with $u$ on $K$. We now bound
    \begin{align*}
	\int_K |u|^2\ d\Sigma^2 \le \int_{\Gamma_+} |u|^2 \ell^2\ d\Sigma^2 = \int_{\Gamma_+} |T(u\ell)|^2\ d\Sigma^2.
    \end{align*}
    From the proof of Lemma \ref{lem:trace}, the latter right-hand side is bounded above by $C_0 \|X (u\ell)\|^2$, which is in turn bounded by $C' \|u\|_{W_X^{1,2}(SM)}^2$, where $C'$ only depends on $C_0$ and the compact set $K$. Lemma \ref{lem:Ktrace} is proved.    
\end{proof}

\begin{definition} \label{def:DX}
    For $u\in W_X^{1,2}(SM)$, we write $u|_{\Gamma_+} = 0$ (resp. $\Gamma_-$) if $u|_K = 0$ for every compact set $K\subset \Gamma_+$ (resp. $\Gamma_-$). 
\end{definition}
The conditions above imply in particular the almost-everywhere vanishing of $u$ on $\Gamma_+$ or $\Gamma_-$, and the linear subspaces $W^{1,2}_{X,\pm}(SM)$ of $W^{1,2}_X(SM)$ defined by 
\begin{align}
    W^{1,2}_{X,\pm} = \{ u\in L^2(SM), \quad Xu \in L^2(SM), \quad u|_{\Gamma_\pm} = 0 \},
    \label{eq:Wspaces}
\end{align}
are natural for considering traces. 
\begin{lemma}\label{lem:traces2}
    Let $T$ as defined in Lemma \ref{lem:trace}. For every $u\in W_{X,-}^{1,2}(SM)$, $Tu(\x,v) = u(\x,v)$ for almost every $(\x,v)\in \Gamma_+$. From the properties of $T$, we can define $u|_{\Gamma_+} \in L^2(\Gamma_+)$, with the continuity estimate
    \begin{align*}
	\|u|_{\Gamma_+}\|^2_{L^2(\Gamma_+)} \le C_0 \|Xu\|^2 \le C_0 \|u\|_{W^{1,2}_X(SM)}^2.
    \end{align*}
\end{lemma}
One can similarly define a trace operator $W_{X,+}^{1,2} (SM) \to L^2(\Gamma_-)$. 
\begin{proof}
    Since $u|_{\Gamma_-} = 0$ a.e., and $\alpha\colon \Gamma_+\to \Gamma_-$ defined by $\alpha(\x,v) := \varphi_{-\tau(\x,-v)}(\x,v)$ is a diffeomorphism\footnote{$\alpha$ is also called the {\em scattering relation} in other contexts.}, then $u|_{\Gamma_-}\circ \alpha = 0$ a.e., hence $Tu$ agrees almost everywhere with $u|_{\Gamma_+}$. 
\end{proof}

\subsection{Proof of Theorem \ref{thm:fwd}} \label{sec:pfTh1}

We view $X$ as an operator with domain $D(X) = W_{X,-}^{1,2}(SM)$. The operator $X$ is accretive: indeed we have for any $u\in D(X)$,
\begin{align*}
    2\text{Re} (Xu,u) &= \int_{SM} ((Xu) \overline{u} + u X \overline{u})\ d\Sigma^3\\
    &= \int_{SM} X(|u|^2)\ d\Sigma^3 \stackrel{\eqref{eq:Green}}{=} \int_{\Gamma_+} |u|_{\Gamma_+}|^2 \mu\ d\Sigma^2 \ge 0.
\end{align*}
The operator $X$ can also easily be seen to be closed. Using Equation \eqref{eq:Green}, the adjoint of $X$ is given by $X^* = -X$, with domain $D(X^*) = W_{X,+}^{1,2}(SM)$, and the calculation
\begin{align*}
    2 \text{Re} (X^* u, u) = - \int_{SM} X (|u|^2)\ d\Sigma^3 = \int_{\Gamma_-} |u|_{\Gamma_-}|^2 |\mu|\ d\Sigma^2\ge 0,
\end{align*}
also shows that $X^*$ is accretive. By \cite[Th.8 p.340]{Dautray5}, the operator $X$ is maximal accretive.

We first show that given $f\in L^2(SM)$, problem \eqref{eq:boltzmann} has a unique solution in $D(X) = W^{1,2}_{X,-}(SM)$ as defined in \eqref{eq:Wspaces}. 

Define the operator $Q\colon L^2(SM) \to L^2(SM)$ by $Q u := a u - Su$. This operator is bounded, self-adjoint, and its operator norm is at most $Q_\infty$ as defined in the statement of the theorem. Moreover, using the subcriticality condition \eqref{eq:subcrit}, we show that 
\begin{align}
    (Qu,u) \ge \delta\|u\|^2.
    \label{eq:Q}
\end{align}
\begin{proof}[Proof of \eqref{eq:Q}] We compute
    \begin{align*}
	(Qu, u) &= \int_M \left[ a(\x) \int_{S^1} |u(\x,\theta)|^2\ d\theta - \int_{S^1\times S^1} k(\x,\theta-\theta') u(\x,\theta) u(\x,\theta')\ d\theta\ d\theta' \right]\ d\x \\
	&\stackrel{\eqref{eq:subcrit}}{\ge} \delta \|u\|^2 + \int_M \left[ k_0(\x) \int_{S^1} |u(\x,\theta)|^2\ d\theta - \int_{S^1\times S^1} k(\x,\theta-\theta') u(\x,\theta) u(\x,\theta')\ d\theta\ d\theta' \right]\ d\x.
    \end{align*}
    To show that the last term is non-negative is a direct consequence of the Cauchy-Schwarz inequality. Indeed, 
    \begin{align*}
	\int_M \int_{S^1\times S^1} &k(\x,\theta-\theta') u(\x,\theta) u(\x,\theta')\ d\theta\ d\theta'\ d\x  \\
	&\le  \left( \int_M \int_{S^1\times S^1} \!\!\!\!\! k(\x, \theta-\theta') |u(\x,\theta)|^2\ d\theta\ d\theta'\ d\x \right)^{\frac{1}{2}}  \left( \int_M \int_{S^1\times S^1} \!\!\!\!\! k(\x, \theta-\theta') |u(\x,\theta')|^2\ d\theta\ d\theta'\ d\x \right)^{\frac{1}{2}} \\
	&= \int_M k_0(\x) \int_{S^1} |u(x,\theta)|^2 \ d\theta\ d\x,
    \end{align*}
    hence \eqref{eq:Q} holds.     
\end{proof}
From \eqref{eq:Q} we deduce that the operator $Q - \delta Id$ is accretive. Then, the operator $X + Q - \delta Id$, with domain of definition $D(X)$ is closed accretive, and its adjoint $(X+Q - \delta Id)^* = -X + Q - \delta Id$, with domain of definition $D(X^*)$ is also accretive. Thus by \cite[Th. 8 p.340]{Dautray5}, the operator $X+Q-\delta Id$ is maximal accretive, and therefore $-(X+Q-\delta Id)$ generates a contraction semigroup of class $C^0$. By \cite[Prop. 1 p.321]{Dautray5}, this implies that for all $\lambda>0$, $(X+Q-\delta Id + \lambda Id) (D(X)) = L^2(SM)$, with resolvent estimate $\|(X+Q-\delta Id + \lambda Id)^{-1}\|_{{\cal L}(L^2(SM))}\le \frac{1}{\lambda}$. In particular, for $\lambda = \delta$, we obtain that there exists a unique solution $u\in L^2(SM)$ such that $Xu + Qu = f$ with $u|_{\Gamma_-} = 0$, and with estimate $\|u\|_{L^2(SM)} \le \frac{1}{\delta} \|f\|_{L^2(SM)}$.

From the transport problem, we then obtain that $Xu = - Qu + f$, whence
\begin{align*}
    \|Xu\| \le \|Qu\| + \|f\| \le Q_\infty \|u\| + \|f\| \le \left(\frac{Q_\infty}{\delta} + 1\right) \|f\|,
\end{align*}
and hence $u\in D(X)$. To conclude, the trace estimate of Lemma \ref{lem:traces2} gives 
\begin{align*}
    \|u|_{\Gamma_+}\|_{L^2(\Gamma_+)} \le \sqrt{C_0} \|Xu\| \le \sqrt{C_0} \left(\frac{Q_\infty}{\delta} + 1\right) \|f\|,
\end{align*}
and this completes the proof of Theorem \ref{thm:fwd}.

\section{Reconstruction} \label{sec:reconstruction}

\subsection{Preliminaries} \label{sec:prelims}

\paragraph{Geometry of $SM$.}

The geodesic vector field $X$ defined in \eqref{eq:X} can be completed into a natural frame of $T(SM)$ by considering the other two vector fields $V:= \partial_\theta$ and $X_\perp$ as defined in \eqref{eq:Xperp}. The frame $\{X,X_\perp, V\}$ has structure equations
\begin{align*}
    [X,V] = X_\perp, \qquad [X_\perp,V] = -X, \qquad [X,X_\perp] = -\kappa(\x)V,
\end{align*}
where $\kappa(\x) := c^2(\x) (\partial_x^2 + \partial_y^2) \log c$ is the Gaussian curvature of the metric $g = c^{-2} \text{id}$. 

\paragraph{Fourier analysis on $SM$.} Recall the Fourier decomposition $L^2(SM) = \oplus_{k\ge 0} H_k$ described Sec. \ref{sec:notation}, with for $k\ge 1$, $H_k = E_k \oplus E_{-k}$. We also denote $\Omega_k = H_k \cap C^\infty(SM)$ for all $k\ge 0$ and $\Lambda_k = E_k \cap C^\infty(SM)$ for all $k\in \Zm\backslash\{0\}$. The free transport operator $X$ can be rewritten as $X = X_+ + X_-$ (see e.g. \cite{Paternain2015}), where $X_\pm (\sob (H_k)) \subset H_{k\pm 1}$ for all $k\ge 0$ with the convention $H_{-1} =\{0\}$. Componentwise, 
\begin{align*}
    X_+ u_k &= \eta_+ u_{k,+} + \eta_- u_{k,-}, \qquad (k\ge 0), \\
    X_- u_k &= \eta_- u_{k,+} + \eta_+ u_{k,-}, \qquad (k >0), \qquad X_- u_0 = 0, \qquad (k=0),
\end{align*}
where we have defined the operators $\eta_{\pm} := (X\pm i X_\perp)/2$ with the property that $\eta_\pm (\Lambda_k) \subset \Lambda_{k\pm 1}$ for all $k\in \Zm$. In coordinates $(\x,\theta)$,
\begin{align*}
    \eta_+ = e^{i\theta} (c(\x) \partial - i (\partial c) \partial_\theta), \qquad \eta_- = \overline{\eta_+}, \qquad \text{where } \quad \partial := \frac{1}{2} (\partial_x - i \partial_y). 
\end{align*}
In what follows, we also denote 
\begin{align*}
    \ker^k \eta_\pm := \Lambda_k \cap \ker \eta_\pm, \qquad k\in \Zm,
\end{align*}
as well as $L^2(\ker^k \eta_\pm)$ the $L^2$ version of it, a closed subspace of $E_k$ as explained in \cite[Sec. 7.1]{Assylbekov2017}. Of special interest will be, for $k\ge 1$ the spaces $\ker^k \eta_- \oplus \ker^{-k} \eta_+$. For $k\ge 2$, such spaces correspond to trace-free, divergence-free tensors of order $k$, while for $k=1$, they correspond to harmonic one-forms. The $L^2$ version will be denoted 
\begin{align}
    \hd_k (M):= L^2(\ker^k \eta_-) \oplus L^2(\ker^{-k} \eta_+), \qquad k\ge 1.
    \label{eq:hdelta}
\end{align}
We have $\hd_k(M) = \{f\in H_k,\ X_- f = 0\}$ for all $k\ge 2$, while for $k = 1$, $\hd_1(M) \subsetneq \{f\in H_1,\ X_- f = 0\}$. 

The relevance of $\hd_k$ comes from the following fact: 

\begin{lemma}\label{lem:decomp}
    Let $k\ge 1$. For any $u \in H_k$, there exists a unique $v\in \sobzero (H_{k-1})$ and $g\in \hd_k$ such that
    \begin{align*}
	u = X_+ v + g.
    \end{align*}
    The decomposition is $L^2(SM)$-orthogonal.
\end{lemma}

\subsection{Recalls from attenuated tensor tomography} \label{sec:atRt}

In the case where scattering is absent, the measurement operator $\M_{a,0}f$ is the so-called {\em attenuated X-ray transform}, usually written $I_a f$ and given explicitly by 
\begin{align*}
    I_a f(\x,v) = \int_{-\tau(\x,-v)}^0 f(\varphi_t(\x,v)) \exp \left( - \int_t^{0} a(\varphi_s(\x,v))\ ds \right) \ dt, \quad (\x,v)\in \Gamma_+,
\end{align*}
where $\tau(\x,v)$ denotes the first nonnegative time $t$ at which $\varphi_t(x,v)\in \Gamma_+$. 

A natural question is whether the recovery of $f$ from $I_a f$ is possible. If $f$ is restricted to be of degree $m$ for some $m>0$, this is referred to as the attenuated tensor tomography problem. Fixing a harmonic cutoff number $m$, one may find that the transform $I_a$ restricted to integrands of degree at most $m$ has a natural kernel: if $p$ is of degree $m-1$ with components in $\sob_0(M)$, then $(X+a)p$ has degree $m$ and satisfies $I_a[ (X+a)p] = 0$. The question then consists in assessing whether this is the only obstruction, and if it is, to change the problem as follows: given $f$ of degree $m$, find a representative $h$ of $f$ modulo the kernel of $I_a$ (i.e., $I_a f = I_a h$) and reconstruct it from $I_a f$. 

The answer to this last question was recently provided in \cite{Krishnan2018}, and the summary of Theorems 1 and 2 there may be read as follows: 

\begin{theorem}[Theorems 1, 2 in \cite{Krishnan2018}]\label{thm:KMM} Let $(M,g)$ a simple Riemannian surface with boundary and $a\in C^\infty(M)$. Then any $f\in L^2(SM)$ of degree $m$ admits a unique decomposition
    \begin{align*}
	f = (X+a) p + h,
    \end{align*}
    where $p$ is of degree $m-1$ with components in $\sob_0(M)$. In addition, $h\in L^2(SM)$ is of degree $m$ and of the form
    \begin{align*}
	h = h_0 + X_\perp h_\perp + \sum_{k=1}^m h_k,
    \end{align*}
    with $h_0\in L^2(M)$, $h_\perp \in \sob_0(M)$ and for $k\ge 1$, $h_k\in H_k^{\text{sol}}(M)$.  

    We thus have $I_a f = I_a h$, and $h$ can be uniquely and constructively recovered from $I_a f$.
    
    If $f\in C^\infty(SM)$, then $p,h_0,h_\perp, h_k$ are all smooth.     
\end{theorem}

Such a result is the basis of the derivations that follow in the next section. Let us mention that, except for the last statement, the conclusions above should be conjectured to hold for $a\in L^\infty(M)$. 

\begin{remark}
    All of the singular information is contained in $h_0$ and $h_\perp$, whereas the $h_k$'s for $|k|\ge 1$ are smooth in $\x\in M^{\text{int}}$ and can be viewed as residual terms. The reconstruction must (i) take care of the analytic terms first, from higher to lower angular dependence, then (ii) the reconstruction of $(h_0,h_\perp)$ must be carried out. Steps (i) and (ii) are fully explicit in the case of the Euclidean disk \cite{Monard2017a}; in the case of simple surfaces/domains, reconstruction formulas for step (ii) are given in \cite{Assylbekov2017}, while step (i) is described in \cite{Krishnan2018}. For this last case, the reconstruction relies on the existence of invariant distributions whose existence is provided through microlocal arguments, though a fully constructive approach remains to be found.
\end{remark}

\subsection{Injective problems - proof of Theorem \ref{thm:main1}} \label{sec:main1}

We now prove Theorem \ref{thm:main1}. We first treat the case of isotropic scattering as it provides a succinct introduction to the approach. 

\subsubsection{Isotropic scattering} 

Suppose the scattering kernel isotropic, i.e., $k(\x,\theta-\theta') = k_0(\x)$. 

\paragraph{Case (1).} Suppose $f = f_0 + X_\perp f_\perp$. In this case, equation \eqref{eq:boltzmann} reads
\begin{align}
  X u + a(\x) u = k_0 u_0  + f_0 + X_\perp f_\perp, \qquad u|_{\Gamma_-} = 0,
  \label{eq:iso}
\end{align}
and where $u|_{\Gamma_+} = I_a [k_0 u_0 + f_0 + X_\perp f_\perp]$. By injectivity of the attenuated ray transform over sums of functions and solenoidal vector fields, we can recover $k_0 u_0 + f_0$ and $f_\perp$ separately from the data, see e.g. \cite{Assylbekov2017,Monard2017a}. In particular, the right-hand-side $\tilde f$ of \eqref{eq:iso} is fully known, thus we may integrate it to recover $u$ namely through the relation $u(\x,v) = \int_{-\tau(\x,-v)}^0 \tilde f(\varphi_t(\x,v))\ dt$. Since $k_0$ is known, then we finally recover $f_0 = (\tilde f)_0 - k_0 u_0$. 

\paragraph{Case (2).} Suppose $f = f_1 \in H_1$ is a vector field. Then the equation $(X+a)u = Su + f$ provides the data $I_a[k_0 u_0 + f_1]$, a form of integrand over which the attenuated X-ray transform is no longer injective. We thus use the Hodge decomposition to rewrite $f = X \tilde f_0 + X_\perp \tilde f_\perp + \omega_1$ for some $\tilde f_0, \tilde f_\perp \in \sobzero(M)$ and $\omega_1 \in H_1^{\text{sol}}(M)$. The transport equation can be rewritten as
\begin{align}
    (X+a) (u-\tilde f_0) = k_0 u_0 -a\tilde f_0 + X_\perp \tilde f_\perp + \omega_1, \qquad (u-\tilde f_0)|_{\Gamma_-} = 0,
    \label{eq:iso2}
\end{align}
and we deduce by direct integration that $\M_{a,k}(f) = I_a [(k_0 u_0 - a\tilde f_0) + X_\perp \tilde f_\perp + \omega_1]$. The integrand inside $I_a$ can now be uniquely reconstructed, see \cite{Assylbekov2017,Monard2017a}, and this determines, separately,
\begin{align*}
    k_0 u_0 - a \tilde f_0, \qquad \tilde f_\perp, \qquad \omega_1.
\end{align*}
Since the right hand side of \eqref{eq:iso2} is known, this determines $u-\tilde f_0$ uniquely through 
\begin{align*}
    (u-\tilde f_0) (\x,v) = \int_{-\tau(\x,-v)}^0 [k_0 u_0 -a\tilde f_0 + X_\perp \tilde f_\perp + \omega_1](\varphi_t(\x,v)) \ dt
\end{align*}
and in particular, $u_0 - \tilde f_0 = (u-\tilde f_0)_0$ is known. We finally reconstruct $\tilde f_0$ by removing $u_0$ from the two reconstructed functions involving it, through the linear combination: 
\begin{align*}
    \tilde f_0 = \frac{1}{k_0 - a} (k_0 u_0 - a\tilde f_0) - k_0 (u_0 - \tilde f_0).
\end{align*} 
That $k_0-a$ vanishes nowhere precisely comes from the subcriticality condition \eqref{eq:subcrit}. Now that $\tilde f_0$, $\tilde f_\perp$ and $\omega_1$ are all reconstructed, we can recover $f$. 

\subsubsection{Scattering kernel of arbitrary finite degree} 

Now suppose $k$ has degree $m$, of the form
\begin{align*}
    k(\x,\theta) = k_0 (\x) + \sum_{n=1}^m (\tilde k_{n,+}(\x) e^{in\theta} + \tilde k_{n,-}(\x) e^{-in\theta}). 
\end{align*}
Call $u$ the unique solution of the transport problem 
\begin{align*}
    Xu + a(\x) u = f(\x) + k_0 u_0 + \sum_{n=1}^m (\tilde k_{n,+} \tilde u_{n,+} e^{in\theta} + \tilde k_{n,-} \tilde u_{n,-} e^{-in\theta}), \qquad u|_{\Gamma_-} = 0,
\end{align*}
and where $\M_{a,k} f = u|_{\Gamma_+}$ is measured and equals the attenuated X-ray transform of the right-hand side. 

In the last section, Case (1) generated a form of integrand over which the operator $I_a$ was injective, while Case (2) required an additional step exploiting a known gauge of $I_a$. The general case is similar to the latter, since $I_a$ has a non-trivial kernel as soon as one considers integrands of nonzero degree. 

\smallskip
\noindent {\bf Step 1. Moving inside the gauge of the attenuated X-ray transform.}
\smallskip

We first modify this transport equation in such a way that the values at the boundary are not modified, yet where the right-hand side will take a form over which the attenuated transform is injective. Namely, by virtue of Theorem \ref{thm:KMM}, the function of degree $m$ given by $f + Su$ admits a unique decomposition
\begin{align*}
    f + Su = (X+a) p + \tilde f,
\end{align*}
where $p$ is of degree $m-1$ with components in $\sobzero(M)$, and where $\tilde f$ can be uniquely reconstructed from $I_a \tilde f = I_a [f+Su] = \M_{a,k}(f)$, as explained in \cite{Assylbekov2017,Monard2017a,Krishnan2018}. 

\smallskip
\noindent {\bf Step 2. Reconstruction of $f$.}
\smallskip
One may now assume that $\tilde f$ has been reconstructed. The crux is now to show how to reconstruct $f$ from $\tilde f$ and the relations written above. The important equalities are
\begin{align}
    Xu + a u = f + Su = (X+a)p + \tilde f,
    \label{eq:3rels}
\end{align}
where $\tilde f$ is known from $\M_{a,k}(f)$. Subtracting the right side from the left and using that $p|_{\partial SM} = 0$, we obtain the transport problem
\begin{align*}
    X(u-p) + a(u-p) = \tilde f, \qquad (u-p)|_{\Gamma_-} = 0.
\end{align*}
Since the right-hand side is known, this determines $(u-p)$ by direct integration, that is, 
\begin{align*}
    (u-p)(\x,v) = \int_{-\tau(\x,-v)}^0 \tilde f(\varphi_t(\x,v))\ dt, \qquad (\x,v)\in SM.
\end{align*}

Now using the middle and right sides of \eqref{eq:3rels}, we may write
\begin{align}
    (X+a) p - Sp - f = S(u-p) - \tilde f.
    \label{eq:temp2}
\end{align}
The right side is known, and assuming that $f$ is of degree at most $1$, we project the equation above onto the subspaces $H_m, \dots, H_2$ to obtain a triangular system of equations determining $p_{m-1}, \dots, p_1$ in descending order. Namely, since $p_k = 0$ for $k\ge m$ and $p|_{\partial SM} =0$, the projections of \eqref{eq:temp2} are: 
\begin{align*}
    H_m:&\qquad X_+ p_{m-1} = (S (u-p))_m - \tilde f_m, \qquad p_{m-1}|_{\partial M} = 0, \\
    H_{m-1}:&\qquad X_+ p_{m-2} + a p_{m-1} - S p_{m-1} = (S (u-p))_{m-1} - \tilde f_{m-1}, \qquad p_{m-2}|_{\partial M} = 0, \\
    H_{m-2}:&\qquad X_+ p_{m-3} + X_-p_{m-1} + a p_{m-2} - S p_{m-2} = (S (u-p))_{m-2} - \tilde f_{m-2}, \qquad p_{m-3}|_{\partial M} = 0, \\ 
    \vdots & \\
    H_3: &\qquad X_+ p_2 + X_- p_4 + a p_3 - S p_3 = (S (u-p))_{3} - \tilde f_{3}, \qquad p_2|_{\partial M} = 0, \\ 
    H_2: &\qquad X_+ p_1 + X_- p_3 + a p_2 - S p_2 = (S (u-p))_{2} - \tilde f_{2}, \qquad p_1|_{\partial M} = 0. 
\end{align*}
Note that the operator $S$ is block-diagonal on the decomposition \eqref{eq:harmonics} so that $(Sp)_k = S p_k$ for all $k\ge 0$. Each equation above is an elliptic problem for $p_k$ of the form 
\begin{align*}
    X_+ p_k = h_{k+1}, \qquad p_k|_{\partial SM} = 0,
\end{align*}
where the right-hand side $h_{k+1}$ is known. More specifically, the equation above gives two first-order elliptic equations 
\begin{align*}
    \eta_+ p_{k,+} = h_{k+1,+}, \qquad p_{k,+}|_{\partial SM} = 0, \qquad \eta_- p_{k,-} = h_{k+1,-}, \qquad p_{k,-}|_{\partial SM} = 0. 
\end{align*}
We now explain how to recover $p_{k,+}$ as the recovery of $p_{k,-}$ is similar\footnote{In fact, if all functions are real-valued, we immediately have $p_{k,-} = \overline{p_{k,+}}$.}. Writing 
\begin{align*}
    p_{k,+} = e^{ik\theta} \tilde p_{k,+} (\x), \qquad h_{k+1,+} = e^{i(k+1)\theta} \tilde h_{k+1,+}(\x), \qquad \eta_+ = e^{i\theta} (c(\x) \partial - i(\partial c) \partial_\theta),
\end{align*}
the equation $\eta_+ p_{k,+} = h_{k+1,+}$ becomes a $\partial$ equation of the form
\begin{align*}
    c^{1-k}\partial (c^k \tilde p_{k,+}) = \tilde h_{k+1,+}, \qquad p_{k,+}|_{\partial M} = 0.
\end{align*}
In coordinates $(\x,\theta)$, the Laplace-Beltrami operator on $M$ is given by $\Delta_g = 4 c^2 \dbar \partial$, and thus we may turn the previous equation into the Poisson problem
\begin{align*}
    \Delta_g \tilde p_{k,+} = 4 c^2 \dbar \left( c^{k-1} \tilde h_{k+1,+} \right), \qquad \tilde p_{k,+}|_{\partial M} =0,
\end{align*}
providing a constructive reconstruction procedure for $p_{k,+}$ from $h_{k+1,+}$.

At this point, we have determined $p_{m-1}, \dots,  p_1$ in descending order. To reconstruct the source, we now finish the proof depending on which type of source is considered. 

\paragraph{Case (1).} Assume $f$ takes the form $f = f_0 + X_\perp f_\perp$ with $f_0 \in L^2(M)$ and $f_\perp \in \sob(M)$. Projecting \eqref{eq:temp2} onto $H_1$ and $H_0$ gives
\begin{align}
    \begin{split}
	\Lambda_1:&\qquad \eta_+ (p_0+if_\perp) + \eta_- p_{2,+} + a p_{1,+} - (Sp)_{1,+} = (S(u-p))_{1,+} - \tilde f_{1,+}, \\
	\Lambda_{-1}:&\qquad \eta_- (p_0-if_\perp) + \eta_+ p_{2,-} + a p_{1,-} - (Sp)_{1,-} = (S(u-p))_{1,-} - \tilde f_{1,-}, \\
	H_0:&\qquad X_- p_1 + \sigma_a p_0 - f_0 = (S(u-p))_0 - \tilde f_0.	
    \end{split}
    \label{eq:proj1}
\end{align}
The first two equations form a closed elliptic system for $p_0$ and $f_\perp$, determining both of them. Namely, this system looks like
\begin{align}
    \eta_+ (p_0 + if_\perp) = s_{1,+}, \qquad \eta_- (p_0 - if_\perp) = s_{1,-},
    \label{eq:p0fperp}
\end{align}
where the right-hand sides $s_{1,\pm} := (S(u-p))_{1,\pm} - \tilde f_{1,\pm} - \eta_\mp p_{2,\pm} - a p_{1,\pm} + (Sp)_{1,\pm}$ are known. Applying $\eta_-$ to the first one, $\eta_-$ to the second, using the fact that on $\Omega_0$, $\eta_+\eta_- = \eta_- \eta_+ = \frac{1}{4} \Delta_g$ ($\Delta_g$: Laplace-Beltrami operator of the domain), we can decouple the above system as: 
\begin{align*}
    \Delta_g p_0 = 2(\eta_- s_{1,+} + \eta_+ s_{1,-}), \qquad \Delta_g f_\perp = \frac{2}{i} (\eta_- s_{1,+} - \eta_+ s_{1,-}). 
\end{align*}
Since $p_0|_{\partial M} = 0$, $p_0$ is uniquely determined by a homogeneous Dirichlet problem, and we now explain how to derive a Neuman condition for $f_\perp$: taking the difference of the equations in \eqref{eq:p0fperp} and using that $\eta_+ + \eta_- = X$ and $\eta_+ - \eta_- = iX_\perp$, we arrive at 
\begin{align*}
    i X_\perp p_0 + i Xf = s_{1,+} - s_{1,-}.
\end{align*}
Given a point $\x\in \partial M$ with outgoing unit normal $\nu_\x$, we now evaluate this equality at $(\x,\nu_\x)$. The term $X_\perp p_0 (\x,\nu_\x)$ is a tangential derivative of $p_0$ along $\partial M$, therefore vanishing, while the term $Xf_\perp (\x,\nu_\x) = \partial_\nu f_\perp(\x)$. We thus obtain the Neuman boundary condition
\begin{align*}
    \partial_\nu f_\perp(\x) = -i(s_{1,+} - s_{1,-}) (\x,\nu_\x), \qquad \x\in \partial M.
\end{align*}
As a conclusion, $f_\perp$ satisfies a second-order elliptic equation with Neuman boundary condition, and is thus determined up to a constant. Finally, $f_0$ is the only unknown in the last equation of \eqref{eq:proj1} and is therefore determined as well.  

\paragraph{Case (2).} Assume $f = f_1\in H_1$ is a vector field. Projecting \eqref{eq:temp2} onto $H_1$ and $H_0$ gives
\begin{align*}
    H_1:&\qquad X_+ p_0 + X_- p_2 + a p_1 - (Sp)_1 - f_1 = (S(u-p))_1 - \tilde f_1, \\
    H_0:&\qquad X_- p_1 + \sigma_a p_0 = (S(u-p))_0 - \tilde f_0.
\end{align*}
This time, the second equation first determines $p_0$ uniquely, since the subcriticality condition \eqref{eq:subcrit} ensures that $\sigma_a$ vanishes nowhere on $M$. Then the first equations give $f_1$ immediately.

The proof of Theorem \ref{thm:main1} is complete. 

\begin{remark}[On the efficiency of the reconstruction procedure] One may notice that there is no differentiation at any moment in this reconstruction procedure, so it is fairly well-behaved. In addition, the inversion process does not require to solve transport equations with scattering terms, which can be a costly step. The only PDE to be solved for is a free transport one. Moreover, there is no need to store three-dimensional structures, as moments $w_n$ are computed one at a time.
\end{remark}

\subsection{Non-injective problems: proof of Theorem \ref{thm:main2}} \label{sec:main2}

Now, sources of the form $f = f_0 + f_1$ are usually not uniquely reconstructible from their attenuated X-ray transform, and the case with scattering is no different. This is even more true for sources of higher degree. In this section, we now explain how to describe the gauge of these non-injective problems when the scattering kernel has finite harmonic content.

\begin{theorem}\label{thm:gauge}
    Suppose that $k$ has degree $m$ and that $f = f_0 + f_1$ is such that $\M_{a,k} (f) = 0$. Then there exists $p\in W_0^{1,2}(M)$ such that $f = (X+\sigma_a) p$.    
\end{theorem}

The proof is based on the solution of the attenuated tensor tomography problem, as described by the following theorem, whose proof is based on \cite[Propositions 4.1 and 4.2]{Paternain2011a} and the use of holomorphic integrating factors as introduced in \cite{Salo2011}. 

\begin{theorem}[Theorem 1 in \cite{Krishnan2018}] \label{lem:onedown}
    Suppose that $(X+a) u = f$ holds on $SM$ with $f$ of degree $m\ge 0$. If $m\ge 1$ and $u|_{\Gamma_\pm} = 0$, then $u$ has degree $m-1$. If $m=0$ and $u|_{\Gamma_\pm} = 0$, then $u = f = 0$.
\end{theorem}

\begin{proof}[Proof of Theorem \ref{thm:gauge}.] Under the assumption of the theorem, there exists $u$ such that 
    \begin{align*}
	(X+a)u = f + Su, \qquad (SM), \qquad u|_{\Gamma_\pm} =0.
    \end{align*}
    Since the right-hand side has degree at most $m$, then by Theorem \ref{lem:onedown}, $u$ has degree at most $m-1$. But then the right-hand side has degree at most $m-1$, and so on, until the right-hand side has degree at most $1$ (since at least $f = f_0 + f_1$). Then by Theorem \ref{lem:onedown}, $u = u_0$ with $u_0|_{\partial M}=0$, and the transport equation reads 
    \begin{align*}
	(X+a)u_0 = f_0 + f_1 + k_0 u_0.
    \end{align*}
    Upon setting $p:= u_0$ and noticing that $a-k_0 = \sigma_a$, the proof follows.     
\end{proof}

Following this idea, we push this further to a source $f$ of arbitrary finite degree. 

\begin{proof}[Proof of Theorem \ref{thm:main2}] $(\impliedby)$ Suppose $f = (X+a-S)p$ for some $p$ of degree $m-1$ with components in $\sobzero(M)$. Then the transport problem
    \[ (X+a) u = f + Su \qquad (SM), \qquad u|_{\Gamma_-} = 0, \]
    can be rewritten as 
    \[ (X+a) (u-p) = S(u-p) \qquad (SM), \qquad (u-p)|_{\Gamma_-} = 0, \]
    so from the forward theory, $u-p$ vanishes identically, including at $\Gamma_+$, and thus 
    \[ \M_{a,k} (f) = u|_{\Gamma_+} = (u-p)|_{\Gamma_+} = 0. \]
    
    $(\implies)$ Suppose $f$ has degree $m$ and that $\M_{a,k} (f) = 0$. Let $n$ be the degree of $k$. Then the solution $u$ of \eqref{eq:transport} vanishes on $\Gamma_\pm$ and we have the relation
    \begin{align}
	(X+a) u = Su + f, \qquad u|_{\Gamma_\pm} = 0.
	\label{eq:last}
    \end{align}
    If $n\le m$, then the right hand side of \eqref{eq:last} has degree $m$, therefore by Theorem \eqref{lem:onedown}, $u$ has degree $m-1$ and we conclude by setting $p=u$.

    If $n>m$, then the right hand side of \eqref{eq:last} has degree $n$, therefore by Theorem \eqref{lem:onedown}, $u$ has degree $n-1$. In particular, the term $Su$ has degree $n-1$ and thus so does the right hand side of \eqref{eq:last}. We can inductively decrease the degree of the right hand side in this fashion, until $n=m$, to be back to the previous case $n\le m$. Theorem \ref{thm:main2} is proved.
\end{proof}

\paragraph{Acknowledgements.} F.M. thanks Plamen Stefanov for fruitful discussions. G.B. was supported by NSF Grant DMS-1908736 and ONR Grant N00014-17-1-2096. F.M. was supported by NSF Grant DMS-1814104 and a UC Hellman Fellowship.


\begin{thebibliography}{10}
\bibitem{Assylbekov2017}
{\sc Y.~M. Assylbekov, F.~Monard, and G.~Uhlmann}, {\em Inversion formulas and
  range characterizations for the attenuated geodesic ray transform}, Journal
  de Math\'ematiques Pures et Appliqu\'ees, 111 (2018), pp.~161--190.
\newblock 

\bibitem{Assylbekov2015}
{\sc Y.~M. Assylbekov and Y.~Yang}, {\em An inverse radiative transfer in
  refractive media equipped with a magnetic field}, The Journal of Geometric
  Analysis, 25 (2015), pp.~2148--2184.

\bibitem{Bal2008b}
{\sc G.~Bal and A.~Jollivet}, {\em Stability estimates in stationary inverse
  transport}, Inverse Probl. Imaging, 2 (2008), pp.~427--454.

\bibitem{Bal2010b}
\leavevmode\vrule height 2pt depth -1.6pt width 23pt, {\em Approximate
  stability estimates in inverse transport theory}, in Biomedical Mathematics,
  Y.~Censor, M.~Jiang, and G.~Wang, eds., 2010.

\bibitem{Bal2007}
{\sc G.~Bal and A.~Tamasan}, {\em Inverse source problem in transport
  equations}, SIAM J. Math. Anal., 39 (2007), pp.~57--76.

\bibitem{Choulli1996}
{\sc M.~Choulli and P.~Stefanov}, {\em Reconstruction of the coefficients of
  the stationary transport equation from boundary measurements}, Inverse
  Problems, 12 (1996), pp.~L19--L23.

\bibitem{Choulli1999}
\leavevmode\vrule height 2pt depth -1.6pt width 23pt, {\em An inverse boundary
  value problem for the stationary transport equation}, Osaka J. Math., 36
  (1999), pp.~87--104.

\bibitem{Dautray5}
{\sc R.~Dautray and J.-L. Lions}, {\em Mathematical Analysis and Numerical
  Methods for Science and Technology: Volume 5 Evolution Problems I}, Springer
  Science \& Business Media, 2010.

\bibitem{Dautray6}
\leavevmode\vrule height 2pt depth -1.6pt width 23pt, {\em Mathematical
  Analysis and Numerical Methods for Science and Technology: Volume 6 Evolution
  Problems II}, Springer Science \& Business Media, 2012.

\bibitem{Fujiwara2019}
{\sc H.~Fujiwara, K.~Sadiq, and A.~Tamasan}, {\em A fourier approach to the
  inverse source problem in an absorbing and anisotropic scattering medium},
  arXiv:1907.07423,  (2019).

\bibitem{Krishnan2018}
{\sc V.~P. Krishnan, R.~K. Mishra, and F.~Monard}, {\em On solenoidal-injective
  and injective ray transforms of tensor fields on surfaces}, Journal of
  Inverse and Ill-posed Problems (to appear),  (2018).
\newblock arXiv:1807.10730.

\bibitem{McDowall2005}
{\sc S.~McDowall}, {\em Optical tomography on simple riemannian surfaces},
  Communications in Partial Differential Equations, 30 (2005), pp.~1379--1400.

\bibitem{McDowall2009}
{\sc S.~McDowall}, {\em Optical tomography for media with variable index of
  refraction}, Cubo, 11 (2009), pp.~71--97.

\bibitem{McDowall2010}
{\sc S.~McDowall, P.~Stefanov, and A.~Tamasan}, {\em Gauge equivalence in
  stationary radiative transport through media with varying index of
  refraction}, Inverse Probl. Imaging, 4 (2010), pp.~151--167.

\bibitem{McDowall2010a}
\leavevmode\vrule height 2pt depth -1.6pt width 23pt, {\em Stability of the
  gauge equivalent in stationary inverse transport}, Inverse Probl., 26 (2010),
  p.~025006.

\bibitem{McDowall2011}
\leavevmode\vrule height 2pt depth -1.6pt width 23pt, {\em Stability of the
  gauge equivalent classes in inverse stationary transport in refractive
  media}, Contemp. Math., 559 (2011), pp.~85--100.

\bibitem{McDowall2004}
{\sc S.~R. McDowall}, {\em An inverse problem for the transport equation in the
  presence of a riemannian metric}, Pacific journal of mathematics, 216 (2004),
  pp.~303--326.

\bibitem{Monard2015}
{\sc F.~Monard}, {\em Inversion of the attenuated geodesic {X}-ray transform
  over functions and vector fields on simple surfaces}, SIAM J. Math. Anal., 48
  (2016), pp.~1155--1177.
\newblock 

\bibitem{Monard2017a}
\leavevmode\vrule height 2pt depth -1.6pt width 23pt, {\em Efficient tensor
  tomography in fan-beam coordinates. {II}: attenuated transforms}, Inverse
  Problems and Imaging, 12 (2018), pp.~433--460.
\newblock 

\bibitem{Paternain2011a}
{\sc G.~Paternain, M.~Salo, and G.~Uhlmann}, {\em Tensor tomography on
  surfaces}, Inventiones Math., 193 (2013), pp.~229--247.
\newblock 

\bibitem{Paternain2015}
{\sc G.~P. Paternain, M.~Salo, and G.~Uhlmann}, {\em Invariant distributions,
  beurling transforms and tensor tomography in higher dimensions},
  Mathematische Annalen, 363 (2015), pp.~305--362.

\bibitem{Salo2011}
{\sc M.~Salo and G.~Uhlmann}, {\em {T}he {A}ttenuated {R}ay {T}ransform on
  {S}imple {S}urfaces}, J. Diff. Geom., 88 (2011), pp.~161--187.

\bibitem{Sharafudtinov1994}
{\sc V.~Sharafutdinov}, {\em Integral geometry of tensor fields}, {VSP},
  Utrecht, The Netherlands, 1994.

\bibitem{Sharafutdinov1994}
{\sc V.~A. Sharafutdinov}, {\em An inverse problem of determining a source in
  the stationary transport equation for a medium with refraction}, Siberian
  Math. J., 35 (1994).

\bibitem{Sharafutdinov1999}
\leavevmode\vrule height 2pt depth -1.6pt width 23pt, {\em The inverse problem
  of determining the source in the stationary transport eqaution on a
  riemanninan manifold}, Journal of Mathematical Sciences, 96 (1999).

\bibitem{Sparr1995}
{\sc G.~Sparr, K.~Strahlen, K.~Lindstrom, and H.~Persson}, {\em Doppler
  tomography for vector fields}, Inverse Problems, 11 (1995), p.~1051.

\bibitem{Stefanov2009a}
{\sc P.~Stefanov and A.~Tamasan}, {\em Uniqueness and non-uniqueness in inverse
  radiative transfer}, Proceedings of the American Mathematical Society, 137
  (2009), pp.~2335--2344.

\bibitem{Stefanov2008a}
{\sc P.~Stefanov and G.~Uhlmann}, {\em An inverse source problem in optical
  molecular imaging}, Analysis \& PDE, 1 (2008), pp.~115--126.

\end{thebibliography}
\end{document}